\documentclass[11pt,a4paper]{article}
\usepackage{amsmath}
\usepackage{amsthm}
\usepackage{amssymb}
\usepackage{amscd}
\usepackage{graphicx}
\usepackage{epsfig}
\usepackage[matrix,arrow,curve]{xy}
\usepackage{color}

\usepackage{bbm}
\usepackage{stmaryrd}
\usepackage{hyperref}
\usepackage[usenames,dvipsnames]{xcolor}

\usepackage{latexsym}
\usepackage{amsfonts}
\input xy
\usepackage{tikz}
\usetikzlibrary{arrows,chains,matrix,positioning,scopes,snakes}
\makeatletter
\tikzset{join/.code=\tikzset{after node path={%
\ifx\tikzchainprevious\pgfutil@empty\else(\tikzchainprevious)%
edge[every join]#1(\tikzchaincurrent)\fi}}}
\makeatother

\tikzset{>=stealth',every on chain/.append style={join},
         every join/.style={->}}
\tikzset{
    >=stealth',
    punkt/.style={
           rectangle,
           rounded corners,
           draw=black, very thick,
           text width=6.5em,
           minimum height=2em,
           text centered},
    pil/.style={
           ->,
           thick,
           shorten <=2pt,
           shorten >=2pt,}
}
\xyoption{all} \tolerance=500
\newtheorem{thm}{Theorem}[section]
\newtheorem{theorem}{Theorem}[section]
\newtheorem{proposition}[thm]{Proposition}
\newtheorem{lemma}[thm]{Lemma}

\theoremstyle{definition}
\newtheorem{example}[thm]{Example}
\newtheorem{remark}[thm]{Remark}
\newtheorem{definition}[thm]{Definition}

\font\black=cmbx10 \font\sblack=cmbx7 \font\ssblack=cmbx5 \font\blackital=cmmib10  \skewchar\blackital='177
\font\sblackital=cmmib7 \skewchar\sblackital='177 \font\ssblackital=cmmib5 \skewchar\ssblackital='177
\font\sanss=cmss12 \font\ssanss=cmss8 scaled 900 \font\sssanss=cmss8 scaled 600 \font\blackboard=msbm10
\font\sblackboard=msbm7 \font\ssblackboard=msbm5 \font\caligr=eusm10 \font\scaligr=eusm7 \font\sscaligr=eusm5

\font\bsymb=cmsy10 scaled\magstep2
\def\all#1{\setbox0=\hbox{\lower1.5pt\hbox{\bsymb
       \char"38}}\setbox1=\hbox{$_{#1}$} \box0\lower2pt\box1\;}
\def\exi#1{\setbox0=\hbox{\lower1.5pt\hbox{\bsymb \char"39}}
       \setbox1=\hbox{$_{#1}$} \box0\lower2pt\box1\;}

\def\tx#1{{\fam0\relax#1}}

\newfam\bifam
\textfont\bifam=\blackital \scriptfont\bifam=\sblackital \scriptscriptfont\bifam=\ssblackital

\newfam\blfam
\textfont\blfam=\black \scriptfont\blfam=\sblack \scriptscriptfont\blfam=\ssblack

\newfam\bbfam
\textfont\bbfam=\blackboard \scriptfont\bbfam=\sblackboard \scriptscriptfont\bbfam=\ssblackboard

\newfam\ssfam
\textfont\ssfam=\sanss \scriptfont\ssfam=\ssanss \scriptscriptfont\ssfam=\sssanss
\def\sss#1{{\fam\ssfam\relax#1}}

\newfam\clfam
\textfont\clfam=\caligr \scriptfont\clfam=\scaligr \scriptscriptfont\clfam=\sscaligr

\def\pmb#1{\setbox0\hbox{${#1}$} \copy0 \kern-\wd0 \kern.2pt \box0}
\def\pmbb#1{\setbox0\hbox{${#1}$} \copy0 \kern-\wd0
      \kern.2pt \copy0 \kern-\wd0 \kern.2pt \box0}
\def\pmbbb#1{\setbox0\hbox{${#1}$} \copy0 \kern-\wd0
      \kern.2pt \copy0 \kern-\wd0 \kern.2pt
    \copy0 \kern-\wd0 \kern.2pt \box0}
\def\pmxb#1{\setbox0\hbox{${#1}$} \copy0 \kern-\wd0
      \kern.2pt \copy0 \kern-\wd0 \kern.2pt
      \copy0 \kern-\wd0 \kern.2pt \copy0 \kern-\wd0 \kern.2pt \box0}
\def\pmxbb#1{\setbox0\hbox{${#1}$} \copy0 \kern-\wd0 \kern.2pt
      \copy0 \kern-\wd0 \kern.2pt
      \copy0 \kern-\wd0 \kern.2pt \copy0 \kern-\wd0 \kern.2pt
      \copy0 \kern-\wd0 \kern.2pt \box0}

\mathchardef\za="710B  %\alpha
\mathchardef\zb="710C  %\beta
\mathchardef\zg="710D  %\gamma
\mathchardef\zd="710E  %\delta
\mathchardef\zve="710F %\epsilon
\mathchardef\zz="7110  %\zeta
\mathchardef\zh="7111  %\eta
\mathchardef\zvy="7112 %\theta
\mathchardef\zi="7113  %\iota
\mathchardef\zk="7114  %\kappa
\mathchardef\zl="7115  %\lambda
\mathchardef\zm="7116  %\mu
\mathchardef\zn="7117  %\nu
\mathchardef\zx="7118  %\xi
\mathchardef\zp="7119  %\pi
\mathchardef\zr="711A  %\rho
\mathchardef\zs="711B  %\sigma
\mathchardef\zt="711C  %\tau
\mathchardef\zu="711D  %\upsilon
\mathchardef\zvf="711E %\phi
\mathchardef\zq="711F  %\chi
\mathchardef\zc="7120  %\psi
\mathchardef\zw="7121  %\omega
\mathchardef\ze="7122  %\varepsilon
\mathchardef\zy="7123  %\vartheta
\mathchardef\zf="7124  %\varomega
\mathchardef\zvr="7125 %\varrho
\mathchardef\zvs="7126 %\varsigma
\mathchardef\zf="7127  %\varphi
\mathchardef\zG="7000  %\Gamma
\mathchardef\zD="7001  %\Delta
\mathchardef\zY="7002  %\Theta
\mathchardef\zL="7003  %\Lambda
\mathchardef\zX="7004  %\Xi
\mathchardef\zP="7005  %\Pi
\mathchardef\zS="7006  %\Sigma
\mathchardef\zU="7007  %\Upsilon
\mathchardef\zF="7008  %\Phi
\mathchardef\zW="700A  %\Omega

\newcommand{\be}{\begin{equation}}
\newcommand{\ee}{\end{equation}}
\newcommand{\raa}{\rightarrow}

\newcommand{\bea}{\begin{eqnarray}}
\newcommand{\eea}{\end{eqnarray}}
\newcommand{\beas}{\begin{eqnarray*}}
\newcommand{\eeas}{\end{eqnarray*}}
\def\*{{\textstyle *}}

\newcommand{\nn}{\nonumber}

\newcommand{\ti}{\times}

\def\lan{\langle}
\def\ran{\rangle}

\def\Aut{\sss{Aut}}
\def\Iso{\sss{Iso}}

\def\ol{\overline}

\def\sT{{\sss T}}

\def\xd{\tx{d}}

\newdir{|>}{%
!/4.5pt/@{|}*:(1,-.2)@^{>}*:(1,+.2)@_{>}}

\newcommand{\dd}{\mathbf{d}}

\newcommand{\N}{\mathbb{N}}

\newcommand{\R}{\mathbb{R}}

\newcommand{\G}{\mathbb{G}}
\newcommand{\V}{\mathbb{V}}
\newcommand{\M}{\mathbb{M}}
\newcommand{\A}{\mathbb{A}}
\newcommand{\Pe}{\mathbb{P}}

\newcommand{\wG}{\widehat{\G}}

\newcommand{\op}[1]{\!\!\mathop{\rm ~#1}\nolimits}

\newcommand{\id}{\op{id}}
\newcommand{\Id}{\mathbbmss{1}}
\newcommand{\Lie}{\textnormal{Lie}}

\newcommand{\catname}[1]{\textnormal{\texttt{#1}}}

\newcommand{\rmA}{\textnormal{A}}

\newcommand{\rmp}{\textnormal{p}}
\newcommand{\rmb}{\textnormal{b}}

\newcommand{\Ad}{\textnormal{Ad}}
\newcommand{\pr}{\textnormal{pr}}
\newcommand{\mg}{\mathfrak{g}}
\newcommand{\bg}{\bar{\mathfrak{g}}}

\def\cG{\mathcal{G}}
\def\cS{\mathcal{S}}

\newcommand{\wu}{\operatorname{deg}}

\begin{document}
\title{\bf On $n$-tuple principal bundles\thanks{Research funded by the  Polish National Science Centre grant under the contract number DEC-2012/06/A/ST1/00256.  }}
\date{}
\author{\\ KATARZYNA  GRABOWSKA$^1$\\ JANUSZ GRABOWSKI$^2$\\
        \\
         $^1$ {\it Faculty of Physics}\\
                {\it University of Warsaw}\\
                \\$^2$ {\it Institute of Mathematics}\\
                {\it Polish Academy of Sciences}
                }
%\author{Janusz Grabowski}

\maketitle

\begin{abstract}{We develop the concept of a double (more generally $n$-tuple) principal bundle departing from
a compatibility condition for a principal action of a Lie group on a groupoid.}
\end{abstract}

\vspace{2mm} \noindent {\bf MSC 2010}: (primary) 22A22, 58H05, 57R22;
 (secondary) 18D05, 18B40.


%\thispagestyle{empty}
%\tableofcontents
\section{Introduction}
The concept of double structures, i.e. two structures of the same type satisfying some compatibility condition, became recently very popular. Let us mention \emph{double vector bundles} \cite{Chen:2014, Konieczna1999,Pr1,Pr2,Pr3}, \emph{double affine bundles} \cite{Grabowski:2010}, \emph{double Lie algebroids and groupoids} \cite{Brown1992, Mackenzie:1992, Mackenzie:2000a}, \emph{double graded bundles} \cite{Grabowski:2012}, etc. Of course, one can consider also $n$-tuple instead of double structures. It is a particular case of mathematical objects consisting with a pair (or $n$ pieces) of compatible (not necessary of the same type) structures, which is crucial for the whole mathematics.

In \cite{Lang:2016} the authors introduced \emph{double principal bundles} (DPBs) as  a diagram of four principal bundles and two exact sequences of their structure groups. The definition is \emph{ad hoc} with no real motivation and explanation. In this paper we introduce DPBs as two principal bundle structures which are naturally compatible in the sense that one of the structures, with the structure group $G_1$, is compatible in a natural sense with the groupoid determined by the second structure, so that we have a $G_1$-groupoid (see \cite{Bruce:2015}). It turns out that the concept of a DPB in this sense is equivalent to the concept of a $\mathbb{G}$-DPB in \cite{Lang:2016}, i.e. the principal bundle structure of a \emph{double principal group} $\G$  -- a Lie group generated by two its normal subgroups $G_1,G_2$.

The advantage of this approach is that it can be easily generalized to the $n$-tuple case with natural examples of $n$-tuple principal groups associated with $n$-tuple vector (or graded) bundles.

\section{$G$-groupoids}
\subsection{The compatibility condition}
Our general reference to the theory of Lie groupoids and Lie algebroids will be Mackenzie's book \cite{Mackenzie2005}.

Let $\mathcal{G} \rightrightarrows M$ be an arbitrary Lie groupoid with \emph{source map} $s: \mathcal{G} \rightarrow M$ and \emph{target map} $t: \mathcal{G} \rightarrow M$. There is also the inclusion map $\iota_M : M \rightarrow \mathcal{G}$, $\zi_M(x)=\Id_x$,   and a \emph{partial multiplication}  $(g,h) \mapsto gh$ which is defined on $\mathcal{G}^{(2)}=\{(g,h)\in\mathcal{G}\ti\mathcal{G}: s(g) = t(h)\}$. Moreover, the manifold $\mathcal{G}$ is foliated by $s$-fibres $\mathcal{G}_{x}= \{ \left.g \in \mathcal{G}\right| s(g) =x\}$, where $x \in M$. As by definition the source and target maps are submersions, the $s$-fibres are themselves smooth manifolds. Geometric objects associated with this foliation will be given the superscript $s$. In particular, the distribution tangent to the leaves of the foliation will be denoted by $\sT^{s} \mathcal{G}$.

Let $\mathcal{G}_{i} \rightrightarrows  M_{i}$ $(i=1,2)$ be a pair of Lie groupoids. Then a \emph{Lie groupoid morphisms} is a pair of maps $(\Phi, \phi)$ such that the following diagram is commutative\smallskip

\begin{tabular}{p{5cm} p{10cm}}
\begin{xy}
(0,20)*+{\mathcal{G}_{1}}="a"; (20,20)*+{\mathcal{G}_{2}}="b";%
(0,0)*+{M_{1}}="c"; (20,0)*+{M_{2}}="d";%
{\ar "a";"b"}?*!/_2mm/{\Phi};
{\ar@<1.ex>"a";"c"} ;
{\ar@<-1.ex> "a";"c"} ?*!/^3mm/{s_{1}} ?*!/_6mm/{t_{1}};
{\ar@<1.ex>"b";"d"};%
{\ar@<-1.ex> "b";"d"}?*!/^3mm/{s_{2}} ?*!/_6mm/{t_{2}};  %
{\ar "c";"d"}?*!/^3mm/{\phi};
\end{xy}
&
\vspace{-60pt}
\noindent in the sense that\smallskip

\noindent $s_{2}\circ \Phi = \phi \circ s_{1},  \hspace{20pt}\textnormal{and} \hspace{20pt} t_{2}\circ \Phi = \phi \circ t_{1} $
\end{tabular}
\smallskip

\noindent subject to the further condition that $\Phi$ respects the (partial) multiplication; if $g,h \in \mathcal{G}_{1}$ are composable, then  $ \Phi(gh) = \Phi(g)\Phi(h)$. It then follows that for $x \in M_{1}$  we have $\Phi(\Id_{x}) = \Id_{\phi(x)}$ and  $\Phi(g^{-1}) = \Phi(g)^{-1}$.

\medskip
In our study of Jacobi and contact groupoids \cite{Bruce:2015} we deal with Lie groupoids that have a compatible action of $\mathbb{R}^{\times}$ upon them; compatibility to be defined shortly. However, as the basic theory of compatible group actions on Lie groupoids is independent of the actual Lie  group, we recall the general setting from \cite{Bruce:2015}. We will use the characterization of a proper action $P\ti G\ni(p,g)\mapsto pg\in P$ of a Lie group $G$ on a manifold $P$ by the condition that for each compact $K\subset P$ the set $K(G)=\{ g\in G\, |\, \ Kg\cap K\ne\emptyset\}$ is relatively compact in $G$ (cf. \cite{Palais:1961}).

\begin{definition} An action $\zr:G\ti\cG\to\cG$ of a Lie group $G$ on a Lie groupoid $\cG\rightrightarrows M$ is called \emph{compatible with the groupoid structure} if $h_g:\cG\to\cG$ are groupoid isomorphisms for all $g\in G$, and \emph{principal} if the action is free and proper. A groupoid equipped with a compatible principal action we call  a \emph{principal bundle $G$-groupoid}, (\emph{$G$-groupoid in short}).
\end{definition}

\begin{remark}
The reader should immediately be reminded of Mackenzie's notion of a \emph{PBG-groupoid} \cite{Mackenzie:1987,Mackenzie:1988} which is close to ours but not exactly the same.
\end{remark}

\begin{proposition} (cf. \cite{Bruce:2015})
A compatible action $\zr:G\ti\cG\to\cG$ of a Lie group $G$ on a Lie groupoid $\cG\rightrightarrows M$ induces a canonical action $\zr_M:G\ti M\to M$ on the manifold $M$ of units. If $\zr$ is principal, then $\zr_M$ is also principal. In particular, in this case the set of orbits $M_0=M/G$ has a canonical manifold structure.
\end{proposition}
\begin{proof}
The action of $G$ on $\mathcal{G}$ commutes with the source and target maps, thus projects onto a $G$-action on the manifold $M$. Moreover, $M$ as an immersed submanifold of $\cG$ is invariant with respect to the $G$-action, and the projected and restricted actions coincide.
As the action of $G$ on $\mathcal{G}$ is principal, it is also principal on the immersed submanifold $M$, so $M$ inherits a structure of a principal $G$-bundle.
\end{proof}

In some cases we have to deal with actions which are originally not free but induce a principal action $[\zr]$ of $G/\ker{\zr}$, where $\ker(\zr)=\{ g\in G\,|\, \ \zr_g=\id\}$ is the kernel of the action. In this case we will speak about a \emph{pre-principal action}.

%Similarly, a $G$-action on a Lie algebroid $A$ is \emph{compatible} if the group acts by Lie algebroid isomorphisms, and we get a \emph{$G$-algebroid} if a principal $G$-action is compatible with the Lie algebroid structure.
\begin{proposition}\label{pro1}
A compatible action $\zr:G\ti\cG\to\cG$ of a Lie group $G$ on a Lie groupoid $\cG\rightrightarrows M$ is pre-principal if and only if the  induced  action $[\zr]_M$ of $G/\ker(\zr)$ on the manifold $M$ of units is principal.
\end{proposition}
\begin{proof}
We already know that if $[\zr]$ is principal, then $[\zr]_M$ is principal, so assume $[\zr]_M$ is principal.
Since $[\zr]_M$ is free, $\zr$ is clearly free, so it remains to show that $\zr$ is proper.
Let $K$ be a compact subset of $P$ and $K_M=\pi(K)$. As $K_M$ is clearly compact, the set
$$(G/\ker(\zr),K_M)=\{ [g]\in G/\ker(\zr)\,|\, \ (K_M[g])\cap K_M\ne\emptyset\}
$$
is compact.
Since $(K[g])\cap K\ne\emptyset$ implies $(K_M[g])\cap K_M\ne\emptyset$,
$$(G/\ker(\zr),K)=\{ [g]\in G/\ker(\zr)\,|\, \ (K[g])\cap K\ne\emptyset\}
$$
is compact.

%Let now $\zs:M\to P$ be a cross section from Lemma \ref{le1}.
\end{proof}

\subsection{Structure of $G$-groupoids}
For any principal $G$ groupoid the reduced manifold $\cG/G=\mathcal{G}_{0}$ is canonically a Lie groupoid $\cG/G=\mathcal{G}_{0} \rightrightarrows M/G=M_{0}$, with the set of units $M_0$, defined by the following structure:

\begin{tabular}{p{5cm} p{5cm}}
\begin{xy}
(0,20)*+{\mathcal{G}}="a"; (20,20)*+{\mathcal{G}_{0}}="b";%
(0,0)*+{M}="c"; (20,0)*+{M_{0}}="d";%
{\ar "a";"b"}?*!/_2mm/{\pi};
{\ar@<1.ex>"a";"c"} ;
{\ar@<-1.ex> "a";"c"} ?*!/^3mm/{s} ?*!/_6mm/{t};
{\ar@<1.ex>"b";"d"};%
{\ar@<-1.ex> "b";"d"}?*!/^3mm/{{\zs}} ?*!/_6mm/{{\zt}};  %
{\ar "c";"d"}?*!/^3mm/{\rmp};
\end{xy}
&\vspace{-80pt}
{\begin{align*}
 &{\zs} \circ \pi = \rmp \circ s\,, \\
 &{\zt} \circ \pi = \rmp \circ t\,,\\
 & \Id_{\rmp(x)}=\zp(\Id_x)\quad\text{for all}\quad x\in M\,,\\
 & \pi(y)^{-1}=\pi(y^{-1})\quad\text{for all}\quad y\in\cG\,,\\
 &\pi(y y') = \pi(y) \pi(y')\quad\text{for all}\quad(y,y') \in \mathcal{G}^{(2)}\,,
\end{align*}}
\end{tabular}

\noindent where $\pi:\cG\to\cG_0$ is the canonical projection.
In fact, the above construction implies, tautologically, that $(\pi, \rmp) : \mathcal{G} \rightrightarrows M \rightarrow \mathcal{G}_{0}\rightrightarrows M_{0}$ is a morphism of Lie groupoids with the above structures.

\begin{theorem}\cite{Bruce:2015} The map
$$
\cS:\cG\to \cG_0\times_{M_0}M:= \left\{(y_0,x) \in \mathcal{G}_{0}\ti M ~ | ~ \rmp(x)= {\zs}(y_0) \right\}\,,\quad
\cS(y)=\left(\zp(y),s(y)\right)\,,
$$
is a diffeomorphism.
With respect to the above identification, the $G$-action is  $(y_0,x)g=(y_0,xg)$, the embedding of units is $\zi_M(x)=(\Id_x,x)$, and the source map reads $s(y_0,x)=x$. As the projection $(y_0,x)\mapsto y_0$ is a groupoid morphism, the groupoid structure is uniquely determined by its target map $t:\cG_0\times_{M_0}M\to M$, $t(y_0,x)=:y_0. x$. On the other hand, such a map $t$ is a target map if and only if it has the following properties (holding for all $x\in M$):
\begin{itemize}
\item[(i)] $\rmp(y_0.x)=\zt(y_0)$ for all $y_0\in\cG_0$\,,
\item[(ii)] $y_0.(y'_0.x)=(y_0y'_0).x$ for all $(y_0,y'_0)\in\cG_0^2$\,,
\item[(iii)] $\Id_{\rmp(x)}.x=x$\,,
\item[(iv)] $y_0.(xg)=(y_0.x)g$ for all $y_0\in\cG_0$ and all $g\in G$\,.
\end{itemize}
\end{theorem}
Note that $(i)-(iii)$ mean that $t$ is an action of $\cG_0$ on $\rmp:M\to M_0$ (c.f. \cite[Definition 1.6.1]{Mackenzie2005}), and $(iv)$ means that the action is $G$-equivariant.
The $G$-groupoid determined by $t$ as above we will denote $\cG_0\times_{M_0}^t M$ and called \emph{$t$-split $G$-groupoid}. Thus, any $G$-groupoid (\ref{pGgr}) is $t$-split for some $t(y_0,x)=y_0.x$ satisfying $(i)-(iv)$.

An important particular case of the above theorem is the case of a trivial principal bundle, $M=M_0\ti G$ which is always a local form of any $G$-groupoid. In this case we can use the identification $\cG_0\times_{M_0}M\simeq\cG_0\times G$ and replace the map $t$ satisfying $(i)$ with a map $\rmb:\cG_0\to G$. Indeed, any map on a Lie group commuting with all the right-translations is a left-translation, so can we write $t(y_0,\zs(y_0),g)=(\zt(y_0),\rmb(y_0)g)$. Now, the properties $(i)-(iv)$ can be reduced to
\be\label{zb}
\rmb(y_0)\rmb(y'_0)=\rmb(y_0y'_0)
\ee
for all $(y_0,y'_0)\in\cG_0^2$, i.e. to
the fact that $\rmb:\cG_0\to G$ is a groupoid morphism.  This is of course always the local form of any $G$-groupoid. The corresponding $G$-groupoid structure, denoted with $\cG_0\ti^\rmb G$, is an obvious generalisation of the groupoid extension by the additive $\R$ with a help of a multiplicative function considered in the literature (cf. \cite[Definition 2.3]{Crainic:2007}), and we have shown that this construction is in a sense universal.
Thus we get the following.
\begin{theorem}\cite{Bruce:2015}\label{trivialsplit} For any $G$-groupoid structure on the trivial $G$-bundle $\cG=\cG_0\ti G$ there is a Lie groupoid structure on $\cG_0$
with the source and target maps $\zs,\zt:\cG_0\to M_0$ and a groupoid morphism $\rmb:\cG_0\to G$ such that the source map $s$, the target map $t$ and the partial multiplication in $\cG$ read
$$s(y_0,g)=(\zs(y_0), g)\,,\quad t(y_0,g)=(\zt(y_0),\rmb(y_0)g)\,,\quad
(y_0,g_1)(y'_0,g_2)=(y_0y'_0,g_2)\,.
$$
\end{theorem}

\section{Double principal bundles}

It is well known that every  principal bundle $\pi:\Pe\to M$ with the structure group $G$ acting
on $\Pe$ by (from the right)
$$\zr:\Pe\ti G\to \Pe\,,\quad (p,g)\mapsto pg\,,$$
induces a canonical Lie groupoid structure on $\cG=(\Pe\ti \Pe)/G$ being the space of orbits $\lan p,q\ran_G$ of
the canonical action of $G$ on $\Pe\ti \Pe$, $((p,q),g)\mapsto (pg,qg)$. The set of units is identified with $M$ embedded by $\zi_M(x)=\lan p_x,p_x\ran_G$, where $p_x$ is any element of $\Pe$ satisfying $\pi(p_x)=x$, and the partial multiplication reads
$$\lan p,q\ran_G\bullet\lan q,r\ran_G=\lan p,r\ran_G\,.$$

If one looks for a compatibility condition of two principal group actions, it is natural to expect that each action should induce a compatible action on the groupoid associated with the other action. Of course, as one should assume that each principal action is compatible with itself, it is impossible to assume that the other group action on the groupoid $(\Pe\ti \Pe)/G$ is principal, since it need not to be free. We will thus use a weaker condition and assume that the action is \emph{pre-principal}.
\begin{definition}
Let
$$\zr':\Pe\ti G'\to \Pe\,,\quad (p,g')\mapsto pg'\,,$$
be a principal action of another Lie group, $G'$, on $\Pe$. We say that the action $\zr'$ is \emph{$(\Pe,G)$-principal} if $\zr'$ induces a compatible pre-principal action of $G'$ on the groupoid $(\Pe\ti \Pe)/G$ by
\be\label{c1} \lan p,q\ran g'=\lan pg',qg'\ran\,.
\ee

We say that a two principal actions $\zr:\Pe\ti G\to \Pe$ and $\zr:\Pe\ti G'\to \Pe$ on the same total space $\Pe$ are \emph{compatible} if the action $\zr'$ is $(\Pe,G)$-principal and \emph{vice versa}, the action $\zr$ is $(\Pe,G')$-principal. A manifold $\Pe$ equipped with two compatible principal
actions we call a \emph{double principal bundle}.
\end{definition}
\begin{remark}
As we will see later, our definition of a double principal bundle is stronger than the definition introduced in \cite{Lang:2016} and  corresponds to the definition of a $\G$-DPB therein. We find this definition better motivated and better suited to the concept of associated bundles.
\end{remark}

\begin{theorem}\label{t1}
If $\Pe$ is a double principal bundle with respect to two principal actions $\zr, \zr'$ of Lie groups $G,G'$, respectively, then there is a canonical structure of a Lie group on $\G=\{ \zr_g\zr_{g'}\,|\, g\in G\,, \ g'\in G'\}$ such that $G\simeq \{ \zr_g\,|\, g\in G\}$  and $G'\simeq\{\zr'_{g'}\,|\, g'\in G'\}$ are closed normal Lie subgroups of $G$.

Moreover, the obvious action of $\G$ on $\Pe$ is principal and induces principal actions of $[G]=G/G_0$ and $[G']=G'/G_0$, with $G_0=G\cap G'$, on $M'=\Pe/G'$ and $M=\Pe/G$, respectively, such that the commutative diagram of canonical maps,
\be\label{dgr}\begin{xy}
(0,20)*+{\Pe}="a"; (20,20)*+{M'}="b";%
(0,0)*+{M}="c"; (20,0)*+{M_{0}}="d";%
{\ar "a";"b"}?*!/_2mm/{\pi'};
{\ar "a";"c"} ?*!/^2mm/{\pi} ;
{\ar "b";"d"}?*!/_3mm/{{[\pi]}} ;  %
{\ar "c";"d"}?*!/_3mm/{[\pi']};
\end{xy}\,,
\ee
with $M_0=\Pe/\G=M/[G']=M'/[G]$, consists of principal bundle morphisms: $G$-principal bundles for the horizontal maps, and
$G'$-principal bundles for the vertical ones.
\end{theorem}
The group $\G$ in the above proposition we will call the \emph{structure group} of the double principal bundle $\Pe$.

\begin{lemma}\label{le1}
If $\pi:\Pe\to M$ is a $G$-principal bundle, then there is a Borel cross section $\zs:M\to \Pe$ of $\pi$ such that $\zs(K)$ is relatively compact if $K\subset M$ is compact.
\end{lemma}
\begin{proof}[Proof of Lemma]
Let $\{ U_i\}_{i\in\N}$ be a locally finite open covering of $M$ such that there is a local smooth cross section $\zs_i:U_i\to \Pe$, $i\in\N$. Define $\zs$ so that
$$\zs(m)=\zs_i(m)\quad \text{for} \quad m\in U_i\setminus \bigcup_{j=0}^{i-1}U_j\,.$$
The section $\zs$ is clearly a Borel section. Any compact $K\subset M$ intersects with only a finite number of $U_i$, say, $U_{i_1},\dots U_{i_k}$. Then $\zs(K)$ is contained in
$\zs_{i_1}(K\cap U_{i_1})\cup\cdots\cup\zs_{i_k}(K\cap U_{i_k})$, thus is relatively compact.

\end{proof}
\begin{remark}
One of the instances of the above result is the existence of the corresponding cross section
for the canonical projection of a Lie group $G$ onto the space $G/H$ of its cosets modulo
a closed subgroup $H$. Extensions to more general cases of groups can be found in \cite{Mackey:1952}, \cite{Palais:1961} and \cite{Kehlet:1984}.
\end{remark}

\begin{proof}[Proof of Theorem \ref{t1}]
Let $(p,q)\in \Pe\ti \Pe$ and $g'\in G'$. According to (\ref{c1}), for each $g\in G$ we have
$$\lan pg',qg'\ran_G=\lan p,q\ran_Gg'=\lan pg,qg\ran_Gg'=\lan pgg',qgg'\ran_G\,.$$
Hence, there is a uniquely determined element $g_{g'}$ such that
$pg'g_{g'}=pgg'$ and $qg'g_{g'}$,. In consequence, for each $g'\in G'$ and each $g\in G$ there is $g_{g'}\in G$ such that, for each $p\in \Pe$, we have
\be\label{c2}
pgg'=pg'g_{g'}\,,
\ee
or equivalently
\be\label{c2a}
q(g')^{-1}gg'=qg_{g'}\,.
\ee
It is now easy to see from (\ref{c2}) that $g'\mapsto(g\mapsto g_{g'})$ is an action of $G'$ on $G$ by group homomorphisms:
\be\label{c3}
g_{g'_1g'_2}=(g_{g'_1})_{g'_2}\quad\text{and} \quad (g_1g_2)_{g'}=(g_1)_{g'}(g_2)_{g'}\,.
\ee
It is also clear from (\ref{c2a}) that the map $G\ti G'\ni(g,g')\mapsto g_{g'}\in G$ is smooth.

Let now ${G'\ltimes G}=G'\ti G$ with the group multiplication
\be\label{c4}
(g',g)(g'_1,g_1)=(g'g'_1,g_{g'_1}g_1)\,.
\ee
The neutral element is the pair of neutral elements $(e',e)$ and the inverse reads
$$(g',g)^{-1}=((g')^{-1},g^{-1}_{(g')^{-1}})\,.$$
It is easy to see that $G$ and $G'$ are closed subgroups generating ${G'\ltimes G}$ and a direct inspection shows that $G$ is a normal subgroup. Actually, if $\Ad$ is the adjoint representation of ${G'\ltimes G}$, then
$\Ad_{g'}g=g_{(g')^{-1}}$ for $g\in G$ and $g'\in G'$.

Of course, changing the roles of $G$ and $G'$, we get an action $g\mapsto(g'\mapsto g'_{g})$ of $G'$ on $G$ by group homomorphisms and we can construct the corresponding group $\widehat{G'}$. Moreover, as
$$pgg'g^{-1}(g')^{-1}=p(g')_{g^{-1}}(g')^{-1}=pgg_{(g')^{-1}}^{-1}\,,
$$
we have the identity
$$(g')_{g^{-1}}(g')^{-1}=gg_{(g')^{-1}}^{-1}\,.$$

We have a canonical action of ${G'\ltimes G}$ on $\Pe$ given by $(g',g)\mapsto \zr'_{g'}\circ\zr_g$.
Indeed, according to (\ref{c2}), $(g',g)(g'_1,g_1)=(g'g'_1,g_{g'_1}g_1)$ acts as
$$\zr'_{g'g'_1}\circ\zr_{g_{g'_1}g_1}=\zr'_{g'}\zr'_{g'_1}\zr_{g_{g'_1}}\zr_{g_1}=
\zr'_{g'}\zr_{g}\zr'_{g'_1}\zr_{g_1}\,.
$$
The kernel the ${G'\ltimes G}$-action is the closed normal subgroup $G_0=\{(g',g)\in G'\ti G\,|\, \zr'_{g'}\circ\zr_g=\id\}$. Put $\G=({G'\ltimes G})/G_0$ and denote the coset of $(g',g)$ with $[g',g]$. Clearly, $\G$ is a Lie group acting on $\Pe$ in the obvious way, and $G,G'$ can be identified canonically with Lie subgroups of $G$. Moreover, $G$ and $G'$ are normal subgroups of $\G$ and
$$\Ad_{g'}g=g_{(g')^{-1}}\,,\quad \text{and}\quad \Ad_{g}g'=g'_{g^{-1}}$$
for $g\in G$ and $g'\in G'$.

Let us see that the action of $\G$ is free. For, let $p_0[g',g]=p_0$ for some $p_0\in \Pe$. This means that $p_0g'=p_0g^{-1}$.
This implies that
$\lan p_0g',p_0g'\ran=\lan p_0g_{-1},p_0g_{-1}\ran=\lan p_0,p_0\ran$ so that $\lan p_0,p_0\ran$ is a fixed point of the action induced on $(\Pe\ti \Pe)/G$ by $G'$. As this action is $(\Pe\ti \Pe)/G$-principal, $g'\in\ker([\zr'])$, i.e. $\lan pg',qg'\ran=\lan p,q\ran$ for all $p,q\in \Pe$, so that $\zr'_{g'}$ acts in fibers of $\pi:\Pe\to M$. In particular,
$$\lan p_0,q\ran=\lan p_0,q\ran g'=\lan p_0g^{-1},qg'\ran\,,$$
so that $qg'=qg^{-1}$ for all $q\in \Pe$, thus $\zr'_{g'}\circ\zr_g=\id$.

Actually, what we have proved is also
\be\label{c5} G_0=\ker([\zr'])=G'\cap G
\ee
where $G'\cap G$ is the intersection of $G'$ with $G$ as subgroups of $\G$. Note that, since $G,G'$ generate $\G$, the subgroup $G_0$ normal in $\G$ and $\G/G\simeq [G']$. Note also that $G$ and $G'$ are closed normal subgroup of $\G$. Indeed, $G$ can be characterized as the set of elements of $\G$ which preserve (closed) fibers of $\pi$. For, if $g'\in G'$ preserves the fibers of $\pi$, then $g'\in\ker(\zr')$, so according to (\ref{c5}) $g'\in G$.

We will prove now that the action of $\G$ on $\Pe$ is principal. Let $K$ be a compact subset
of $\Pe$ and $K_M=\pi(K)$ and $K_{M'}=\pi'(K)$ be its projection onto $M$ and $M'$, respectively.
As $K_{M}$ is compact and the $[G']$-action on $M$,
\be\label{c6}\pi(p)[g']=\pi(pg')\,,
\ee
is principal,
$$K_M([G'])=\{[g']\in[G']\,|\, \ (K_{M})[g']\cap K_{M}\ne\emptyset\}
$$
is compact. According to Lemma, there is a Borel cross section $\zs$ of $\pr:G'\to[G']=G'/G_0$
such that ${K^\zs_M}([G'])=\zs(K_M([G']))$ is relatively compact. We can also assume that $\zs([e'])=e'$, Hence, the closure $K'$ of
$KK^\zs_M([G'])$ is compact in $\Pe$. Consequently, as the action of $G$ is principal,
$$(K\cdot K^\zs_M)([G'])=\{[g]\in[G]\,|\, \ (K\cdot K^\zs_M([G']))[g]\cap (K\cdot K^\zs_M([G']))\ne\emptyset\}
$$
is compact, thus
$$K^\zs_M([G'])\cdot (K\cdot K^\zs_M)([G'])\subset\G$$
is compact.

Suppose that $K[g',g]\cap K\ne\emptyset$. Since $\pi(pg'g)=\pi(pg')=\pi(p)[g']$, we have
$$K_{M}[g']\cap K_{M}\ne\emptyset\,,
$$
so that $[g']\in K_M([G'])$. In consequence, there is $g_0\in G_0$ such that
$g'=\zs([g'])g_0$ and
$$Kg'g=K\zs([g'])(g_0g)\subset KK^\zs_M([G'])(g_0g)$$
intersects $K$, thus $KK^\zs_M([G'])$. Hence,
$g_0g\in (K\cdot K^\zs_M)([G'])$ and
$$g'g=\zs([g'])(g_0g)\in K^\zs_M([G'])\cdot (K\cdot K^\zs_M)([G'])\,.$$
This shows that
$$(\G,K)=\{[g',g]\in\G\,|\, \ K[g',g]\cap K\ne\emptyset\}$$ is a subset of $K^\zs_M([G'])\cdot (K\cdot K^\zs_M)([G'])$, thus compact, and proves that the $\G$-action on $\Pe$ is proper.

Finally, as $\zp(p[g',g])=\pi(pg'g)=\pi(p)[g']$ and $\zp'(p[g',g])=\pi'(pg'g)=\pi'(pg')[g]=\pi'(p)[g]$, the projections $\pi,\pi'$ are morphisms of the corresponding principal bundles.

\end{proof}

\section{Double principal groups and their actions}
The group $\G$ which appears in the formulation of Theorem \ref{t1} is a Lie group with two distinguished closed normal subgroups (which are automatically Lie subgroups) $G,G'$ such that $G\cup G'$ generates $\G$. The triples $\wG=(\G;G,G')$ we will call \emph{double principal groups}. This concept coincides with the concept of a \emph{double Lie group} in \cite{Lang:2016}. Indeed, the triple of Lie groups $\widehat{\G}=(\G;G,G')$ as above induces the
exact sequence of group homomorphisms
\be\label{c7} 1\raa G_0\raa\G\overset{\phi}{\raa} [G]\ti[G']\,,
\ee
where $G_0=G\cap G'$, $[G]=G/G_0$, and $[G']=G/G_0$. Conversely, the exact sequence (\ref{c7})
gives rise to the triple $\left(\G,\phi^{-1}(\{ e\}\ti[G']),\phi^{-1}([G]\ti \{ e\})\right)$.
We prefer the term ``double principal group" to prevent any confusion with double Lie groups of Lu and Weinstein \cite{Lu:1990} (cf. also \cite{Mackenzie:1992}) or Drinfel'd doubles.

By a \emph{morphism of double principal groups} $\wG_i=(\G_i,G_i,G'_i)$, $i=1,2$, we understand a group homomorphism $\phi:\G_1\to \G_2$ such that $\phi(G_1)\subset G_2$ and $\phi(G'_1)\subset G'_2$. A double principal group is called \emph{vacant} if the normal subgroups $G$ and $G'$ intersect trivially, i.e. $G\cap G'=\{ e\}$.

\begin{proposition}
A double principal group $\widehat{\G}=(\G;G,G')$ is vacant if and only if the map
$$m: G\ti G'\ni(g,g')\mapsto gg'\in\G$$ is a diffeomorphism.
\end{proposition}
\begin{proof} If $G\cap G'=\{ e\}$, then the map $m$ is clearly smooth and injective. As it is also surjective, it is a smooth bijection. It remains to show that $m$ is a local diffeomorphism.
Let $\bg=\mg\oplus\mg'$ be the decomposition of the Lie algebra $\bg$ of $\G$ into a direct sum of the Lie algebras of $G$ and $G'$. Let $X\in\mg$, $X'\in\mg'$ and $\widehat{X}$ and $\widehat{X'}$ be the corresponding left-invariant vector fields on $\G$. We have
\beas &D_{(g,g')}m\left(\widehat{X}(g),\widehat{X'}(g')\right)=\frac{\xd}{\xd t}_{|t=0}(g\exp(tX)g'\exp(tX'))=\\
&=\frac{\xd}{\xd t}_{|t=0}(gg'\exp(t\Ad_{(g')^{-1}}X)\exp(tX'))
=(gg')_*({\Ad_{(g')^{-1}}X+X'})\,.
\eeas
As $X+X'\mapsto \Ad_{(g')^{-1}}X+X'$ is an isomorphism of $\bg$ for any $g'\in G'$, this shows that the derivative $D_{(g,g')}m$ of $m$ is an isomorphism, thus $m$ is a diffeomorphism.
The converse is trivial.

\end{proof}

\begin{theorem}
Let $\widehat{\G}=(\G;G,G')$ be a double principal group. Then, any principal action $r:\Pe\ti\G\to \Pe$ induces on $\Pe$ a double principal structure relative to the actions $\zr=r_{|G}$, $\zr'=r_{|G'}$.
\end{theorem}
\begin{proof}
Of course, the principal action of $\G$ on $\Pe$ induces principal actions of closed subgroups,
so that $\zr,\zr'$ are principal. Since $G,G'$ are normal subgroup, the adjoint action of $\G$ on itself induces the `dressing actions' of $G$ on $G'$ and \emph{vice versa}:
$$G\ti G'\ni(g,g')\mapsto g'_g=\Ad_{g^{-1}}g'\in G'\,\quad
G'\ti G\ni(g',g)\mapsto g_{g'}g=\Ad_{(g')^{-1}}g\in G\,,
$$
so that
\be\label{c8}
gg'=g'g_{g'}=g'_{g^{-1}}g\,,\quad g'g=gg'_g=g_{(g')^{-1}}g'\,.
\ee
In particular, as $\G$ is generated by $G\cup G'$,
$$\G=GG'=\{ gg'\,|\, \ g\in G\,,g'\in G'\}\quad\text{and}\quad\G=\{ g'g\,|\, \ g\in G\,,g'\in G'\}=G'G\,.
$$
Identities (\ref{c8}) imply also that $\zr$ and $\zr'$ induce canonical actions of $G$ and $G'$ on
$(\Pe\ti \Pe)/G'$ and $(\Pe\ti \Pe)/G$, respectively. The kernels of these actions coincide with $G_0=G\cap G'$, so that we get actions of $[G]=G/G_0=\G/G'$ and $[G']=G'/G_0=\G/G$ by
$$\lan p,q\ran_{G'}[g]=\lan pg,qg\ran_{G'}\,,\quad
\lan p,q\ran_G[g']=\lan pg',qg'\ran_G\,.
$$
These actions are free. Indeed, if $\lan p,q\ran_G[g']=\lan p,g\ran_G$, then there is $g\in G$
such that $pg'=pg$, $qg'=qg$. In particular, $p$ is a fixed point of $g'g^{-1}\in\G$, thus
$g'=g\in G_0$. Similarly we prove that the action of $[G']$ is free.

The action is clearly compatible with the groupoid structure: if $\lan p,q\ran_{G'}$ and $\lan p',q'\ran{G'}$ are composable, then there is $g'\in G'$ such that $q=p'g'$, so $qg=p'g'g=p'gg'_g$
and $\lan p,q\ran_{G'}[g]$ and $\lan p',q'\ran_{G'}[g]$ are composable. Moreover,
$$\lan p,q\ran_{G'}[g]\bullet\lan p',q'\ran_{G'}[g]=\lan pg,q'gg'_g\ran_{G'}=
\left(\lan p,q\ran_{G'}\bullet\lan p',q'\ran_{G'}\right)[g]\,.
$$

It remains to show that the actions are proper. Due to symmetry of conditions, it is enough to prove that the action of $[G]$ on $(\Pe\ti \Pe)/G'$ is proper. Note that $(\Pe\ti \Pe)/G'$ is a principal bundle over $M'=\Pe/G'$ and the induced action of $[G]$ on $M'$ is the same as induced from the action of $G$ on $\Pe$:
$$\pi'(p)[g]=\pi'(pg)\,.
$$
Clearly, it is enough to prove that the latter action is proper.
Let $K'$ be a compact subset of $M'$ and $\zs:M'\to \Pe$ be a section from Lemma \ref{le1}.
The set $K=\zs(K')$ is relatively closed on $\Pe$, so $K(\G)=\{ gg'\in\G\, |\, \ Kgg'\cap K\ne\emptyset\}$ is relatively compact in $\G$. If $[g]\in K'([G])=\{ [g]\in[G]\, |\, \ K'[g]\cap K'\ne\emptyset\}$, then there is $g'\in G'$ such that $gg'\in K(\G)$. Hence, $K'([G])$ is contained in $\pr(K(\G))$, where $\pr:\G\to[G]=\G/G'$ is the canonical projection. As $\pr(K(\G))$ is relatively compact, $K'([G])$ is relatively compact.

\end{proof}

\begin{definition} Double principal bundles described by the above theorem we will call $\widehat{\G}$-principal bundles. A \emph{morphism of $\widehat{\G_i}$-principal bundles $\Pe_i$}, $i=1,2$, is a smooth map $\Phi:\Pe_1\to\Pe_2$ which is equivariant with respect to a morphism of double principal groups $\phi:\G_1\to\G_2$, i.e. $\Phi(p\bar g)=\Phi(p)\phi(\bar g)$.
\end{definition}

In general, any $\G$-manifold, i.e. a manifold $\Pe$ with an action of $\G$, gives rise to the commutative diagram of maps between the spaces of orbits:
\be\label{dgr1}\begin{xy}
(0,20)*+{\Pe}="a"; (20,20)*+{\Pe/G'}="b";%
(0,0)*+{\Pe/G}="c"; (20,0)*+{\Pe/\G}="d";%
{\ar "a";"b"}?*!/_2mm/{\pi'};
{\ar "a";"c"} ?*!/^2mm/{\pi} ;
{\ar "b";"d"}?*!/_3mm/{{[\pi]}} ;  %
{\ar "c";"d"}?*!/_3mm/{[\pi']};
\end{xy}\,.
\ee
Conversely, if we have a commutative diagram $\widehat{\Pe}$ of (locally trivial) smooth fibrations
\be\label{dgr2}\widehat{\Pe}=\quad\begin{gathered}{\begin{xy}
(0,20)*+{\Pe}="a"; (20,20)*+{P'}="b";%
(0,0)*+{P}="c"; (20,0)*+{P_0}="d";%
{\ar "a";"b"}?*!/_2mm/{\pi'};
{\ar "a";"c"} ?*!/^2mm/{\pi} ;
{\ar "b";"d"}?*!/_3mm/{{[\pi]}} ;  %
{\ar "c";"d"}?*!/_3mm/{[\pi']};
\end{xy}}\end{gathered}\,,
\ee
on which $\G$ acts (from the left) in such a way that $G$ acts in fibers of $\pi$ and $G'$ acts in fibers of $\pi'$ (thus $G'$ acts on $P$ and $G$ acts on $P'$), then
we call $\widehat{\Pe}$ (or simply ${\Pe}$) a \emph{$\widehat{\G}$-fibered manifold}.
If $P_0$ is just one point, we speak about a \emph{$\widehat{\G}$-fibered space}.
\begin{example}
Consider a \emph{trivial $\wG$-fibered space} $\Pe=P\ti P'\ti P^{(1,1)}$ with coordinates $(y,y',z)$ and fibrations
$$\pi(y,y',z)=(y')\,,\quad \pi'(y,y',z)=(y)$$
on which the double principal group $\wG=(\G;G,G')$ acts by
$$\bar g(y,y',z)=(A(\bar g)(y),A'(\bar g)(y'),B(\bar g)(y,y',z))$$
such that $A(\bar g)$ is $\id$ for $\bar g\in G$ and $A'(\bar g)$ is $\id$ for $\bar g\in G'$.
\end{example}
\begin{example}\label{ex1}
More specifically, define for $\ol{d}=(d,d',d^0)\in\N^3$ the trivial double space $\A=\R^{\ol{d}}_{aff}=\R^{d}\ti \R^{d'}\ti \R^{d^0}$ with canonical projections $\zr:\A\to\R^d$ and $\zr':\A\to\R^{d'}$, which is viewed as a trivial \emph{double affine space} (cf. \cite{Grabowski:2010}), i.e. we view $\R^{d}, \R^{d'}, \R^{d^0}$ are affine spaces with affine coordinates $(y,y',z)$ and $\G$ is the group $\G=\Aut(\R^{\ol{d}}_{aff})$ of double affine space automorphisms, i.e. diffeomorphisms of the form
\bea\nn\bar g(y,y',z)&=&( \za^0_j(\bar g)+\za^i_j(\bar g)y_i,\,(\za')^0_b(\bar g)+(\za')^a_b(\bar g)y'_a,\\
&&
\zb^{00}_v(\bar g)+\zb^{i0}_v(\bar g)y_i+\zb^{0a}_v(\bar g)y'_a
+\zb^{ia}_v(\bar g)y_iy'_a+\zs^u_v(\bar g)z_u)\,.
\eea
Here $G$ and $G'$ act by automorphisms of the affine fibrations $(y,y',z)\to(y')$ and $(y,y',z)\to(y)$ which act trivially on $y$ and $y'$, respectively.
\end{example}
\begin{example}\label{ex2}
Similarly we can consider the case of a \emph{double vector space} $\V$ \emph{of dimension $\ol{d}$}, $\V=\R^{\ol{d}}_{vect}=\R^{d}\ti \R^{d'}\ti \R^{d^0}$, and its group of automorphisms $\G=\Aut(\R^{\ol{d}}_{vect})$ which acts by
\be\bar g(y,y',z)=\left(\za^i_j(\bar g)y_i,\,(\za')^a_b(\bar g)y'_a,\,
\zb^{ia}_v(\bar g)y_iy'_a+\zs^u_v(\bar g)z_u\right)\,,
\ee
with the action of the subgroups $G$ and $G'$ given by
\be g(y,y',z)=\left(y,\,(\za')^a_b(g)y'_a,\,
\zb^{ia}_v(g)y_iy'_a+\zs^u_v(g)z_u\right)\,,
\ee
and
\be g'(y,y',z)=\left(\za^i_j(g')y_i,\,y',\,
\zb^{ia}_v(g')y_iy'_a+\zs^u_v(g')z_u\right)\,.
\ee
Note that in this case the normal subgroup $G_0=G\cap G'$ of $\Aut(\V)=\Aut(\R^{\ol{d}}_{vect})$ contains another normal subgroup consisting of \emph{statomorphism}, i.e. the automorphism with all linear part vanishing:
\be g_0(y,y',z)=\left(y,\,y',\,
\zb^{ia}_v(g_0)y_iy'_a+z_v\right)\,.
\ee
For the theory of double vector bundles as vector bundles in the category of vector bundles we refer to the original papers by Pradines \cite{Pr1,Pr2,Pr3} (see also \cite{Konieczna1999,Mackenzie:1992}). A substantially simplified approach via \emph{homogeneity structures} can be found in \cite{Grabowski:2012}.
\end{example}

\section{$n$-tuple principal groups and bundles}
Our definition of double principal group and double principal bundle can be generalized to an $n$-tuple case, $n>2$.
The definition is inductive.

\begin{definition}
An \emph{$n$-tuple principal group} is a system $\wG=(\G;G^1\dots,G^n)$, where $\G$ is a Lie group and $G^i$, $i=1,\dots, n$ are its closed normal subgroups such $\cup_{i=1}^nG^i$ generates $G$ and any system $$\wG_i=(G^i;G^1\cap G^i,\dots,G^{i-1}\cap G^i,G^{i+1}\cap G^i,\dots,G^n\cap G^i)$$ is an $(n-1)$-tuple principal group. An \emph{$n$-tuple principal bundle} with the structure $n$-tuple principal group $\wG$ is a manifold $P$ equipped with a principal action of $\G$.
\end{definition}

Suppose $n\ge 2$. From the above definition we obtain easily the following.
\begin{proposition}
If $\wG=(\G,G^1\dots,G^n)$ is an $n$-tuple principal group, then any
triple $\wG^{ij}=(\G,G^i,G^j)$, $i\ne j$, is a double principal group.
Any $n$-tuple principal bundle $P$ with the structure group $\wG$ is a double principal bundle with respect to the action of each double principal group $\wG^{ij}$, $i\ne j$. Moreover, the bundle $P$  is a principal fibration with respect to all projections $\pi_i:P\to P_i=P/G^i$, where each $P_i$ is itself an $(n-1)$-tuple principal bundle with the structure group $\wG_i$.
\end{proposition}
\begin{proof}
It is enough to prove that, for each $i,j=1,\dots, n$, $i\ne j$, the set  $G^i\cup G^j$ generate $\G$.
For $n=2$ this is true by definition, so by the inductive assumption, each $G^l$ is generated by $G^i\cup G^j$.
Because $\cup_{l=1}^nG^l$ generate $\G$, the proposition follows.
\end{proof}
\begin{example}
Example \ref{ex2} can be generalized to the $k$-tuple case if we start with a \emph{$k$-tuple vector bundle bundle}, i.e.
a manifold $E$ equipped with coordinate systems
$(y_a^\zs)_{\zs\in\{ 0,1\}^k_*}$ which yield diffeomorphisms of members of an open cover of $E$ onto $U\ti\Pi_{\zs\in\{ 0,1\}^k}\R^{d_\zs}$, where $U\subset\R^{d_{\mathbf 0}}$,
and transformation rules
$$x'^{\mathbf 0}_a=\phi_a(x^{\mathbf 0})\,,\quad x'^{\zs}_b=\sum_{\zs=\zs_1+\cdots+\zs_m}g_{b,\zs_1,\dots,\zs_m}^{b_1,\dots,b_m}(x^{\mathbf 0})x^{\zs_1}_{b_1}\cdots x^{\zs_m}_{b_m}\,.
$$
Here, ${\mathbf 0}=(0,\dots,0)\in\{ 0,1\}^k$ and $\{ 0,1\}^k_*=\{ 0,1\}^k\setminus\{\mathbf 0\}$ (cf. \cite{Grabowski:2006}). It is easy to see that $E$ is a polynomial fibration over some manifold $M$ (with local coordinates $(x^{\mathbf 0}_a)$) and fibers $E_{x^{\mathbf 0}}$ being  \emph{$k$-tuple vector bundle spaces} (with local coordinates $(x^\zs_a)_{\zs\ne{\mathbf 0}}$). All fibers are isomorphic $k$-tuple vector spaces,
\be\label{ktvb}E_{x^{\mathbf 0}}\simeq E_0=\Pi_{\zs\in\{ 0,1\}^k_*}\R^{d_\zs}
\ee
with the automorphism group $\G=\Aut(E_0)$ acting by
$$x'^{\zs}_b=\sum_{\zs=\zs_1+\cdots+\zs_m}g_{b,\zs_1,\dots,\zs_m}^{b_1,\dots,b_m}x^{\zs_1}_{b_1}\cdots x^{\zs_m}_{b_m}\,,\quad \zs\ne{\mathbf 0}\,.
$$
It is easy to see that $\G$ acts linearly in factors $\R^{d_{\ze_i}}$ in $\Pi_{\zs\in\{ 0,1\}^k_*}\R^{d_\zs}$, $\ze_i=(0,\dots,0,1,0,\dots,0)\in\{ 0,1\}^k$. The following observation is trivial.
\begin{proposition}\label{p1}
The collection $\wG=(\G,G^1\dots,G^k)$, where $G^i$ is the subgroup of $\G$ acting identically on the factor $\R^{d_{\ze_i}}$ in $\Pi_{\zs\in\{ 0,1\}^k_*}\R^{d_\zs}$, is a $k$-tuple principal group.
\end{proposition}
\end{example}

\section{Associated bundles and principal connections}
\subsection{Associated double bundles}
Let now $\widehat{\G}=(\G;G,G')$ be a double principal group, $\Pe$ be a double principal bundle with the structure group $\wG$:
\be\label{dgr2a}\widehat{\Pe}=\quad\begin{gathered}{\begin{xy}
(0,20)*+{\Pe}="a"; (20,20)*+{P'}="b";%
(0,0)*+{P}="c"; (20,0)*+{P_0}="d";%
{\ar "a";"b"}?*!/_2mm/{\pi'};
{\ar "a";"c"} ?*!/^2mm/{\pi} ;
{\ar "b";"d"}?*!/_3mm/{{[\pi]}} ;  %
{\ar "c";"d"}?*!/_3mm/{[\pi']};
\end{xy}}\end{gathered}\,,
\ee
and
\be\label{dgr3}\widehat{\M}=\quad\begin{gathered}{\begin{xy}
(0,20)*+{\M}="a"; (20,20)*+{M'}="b";%
(0,0)*+{M}="c"; (20,0)*+{\{ *\}}="d";%
{\ar "a";"b"}?*!/_2mm/{\zr'};
{\ar "a";"c"} ?*!/^2mm/{\zr} ;
{\ar "b";"d"}?*!/_3mm/{{[\zr]}} ;  %
{\ar "c";"d"}?*!/_3mm/{[\zr']};
\end{xy}}\end{gathered}\,,
\ee
be a $\wG$-fibered space. Then, it is easy to see that the associated bundle
$\Pe\ti_{\wG}\M$ is canonically a $\wG$-fibered manifold, namely

\be\label{dgr4}\widehat{\Pe}\ti_\G\widehat{\M}=\quad\begin{gathered}{\xymatrix{
\Pe\ti_\G\M\ar[rr]^{\Pi'}\ar[dd]_{\Pi}&&{P\ti_{G'}M}\ar[dd]^{[\Pi]}\\ &&\\
P'\ti_GM'\ar[rr]^{[\Pi']}&&P_0}}\end{gathered}
\,,
\ee
where
$$\Pe\ti_\G\M=\left\{[p,u]\,|\,\ [p,u]=\{(p\bar g,\bar g^{-1}u)\,|\,\ \bar p\in\Pe\,, u\in\M\,, g\in\G\}\right\}\,,
$$
and $\Pi'([p,u]=[\pi'(p),\zr'(u)]$, \ $\Pi([p,u])=[\pi(p),\zr(u)]$. Note that we can make this construction starting from a $\G_1$-fibered space and a faithful morphism $\zt:\G\to\G_1$ of double principal bundles to produce the $\G$-action on $\M$ and the associated bundle $\Pe\ti_\G\M=\Pe_\zt\M$.
\begin{example}
If $\M$ is a $\G$-fibered space, fibered over $M$ and $M'$ and  $\Pe$ is a $\G$-principal bundle, then the associated bundle $\Pe\ti_\G\M$ is a $\wG$-fibered manifold which is a locally trivial fibration over $P_0$, locally diffeomorphic to $U\ti\M$ with the same transition functions as the $\G$-principal bundle $\Pe$ by maps $\zs:(U\cap U')\to\G$. Indeed, we apply the classical result on the associated bundle construction to the $\G$-principal bundle $\Pe$. The added value is only the double fibration (\ref{dgr4}) on $\Pe\ti_\G\M$.
\end{example}
\subsection{Frame bundle of a double vector bundle}
An inverse construction is the following.
\begin{example} [Frame bundle construction]
Let
\be\label{dgr2b}\widehat{{\V}}=\quad\begin{gathered}{\begin{xy}
(0,20)*+{\V}="a"; (20,20)*+{V'}="b";%
(0,0)*+{V}="c"; (20,0)*+{V_0}="d";%
{\ar "a";"b"}?*!/_2mm/{\zr'};
{\ar "a";"c"} ?*!/^2mm/{\zr} ;
{\ar "b";"d"}?*!/_3mm/{{[\zr]}} ;  %
{\ar "c";"d"}?*!/_3mm/{[\zr']};
\end{xy}}\end{gathered}\,,
\ee
be a double vector bundle modelled on the double vector space $\R^{\ol{d}}_{vect}$ and let
$\G=\Aut(\R^{\ol{d}}_{vect})$ which is canonically a double principal group (see Example \ref{ex2}). Then, for any $x\in V_0$ the space $\V_x=(\zr'\circ[\zr])^{-1}(\{ x\})$ is a double vector space of dimension $\ol{d}$ and the space $\Iso(\V_x,\R^{\ol{d}}_{vect})$ of isomorphisms of double vector spaces is is a $\G$-principal space with the obvious (right) action $\psi \bar g=\psi\circ \bar g$. The corresponding $\G$-bundle $\Pe=\Iso(\V,\R^{\ol{d}}_{vect})$ over $V_0$ with fibers $\Iso(\V_x,\R^{\ol{d}}_{vect})$ is canonically a double principal bundle with the structure group $\G$. We call it the \emph{frame bundle of the double vector bundle $\V$}. Like in the classical case, one can show that the $\V$ is canonically isomorphic with the associated bundle $\Pe\ti_\G\R^{\ol{d}}_{vect}$. This construction can be generalized to any case when the double fibration (\ref{dgr2b}) is locally trivial modelled on a fibre with a geometric structure admitting a Lie group as the automorphism group.
\end{example}
The mutually converse constructions, that of an associated bundle and the frame bundle, lead to the following `double analog' of a classical result.
\begin{theorem}\label{t6}
Given a double vector space $\R^{\ol{d}}_{vect}$, we have the double principal group $\G=\Aut(\R^{\ol{d}}_{vect})$of its automorphisms. Then the associated bundle construction gives
rise to an equivalence from the category of principal $\G$-bundles on $\Pe_0$ to the category
of double vector bundles of dimension $\ol{d}$ on $P_0$.
\end{theorem}
Actually, the above theorem and constructions have full analogues in the $k$-tuple case.
\subsection{Higher double graded bundles}
All the above constructions for double vector bundles can be generalized to
higher \emph{graded bundles}.

Consider now a smooth action $h:\R\times F\to F$ of the monoid $(\R,\cdot)$ on a manifold $F$ and assume that $h_0(F)=0^F$ for some element $0^F\in F$. Such an action we will call a \textit{ homogeneity structure}. The set $F$ with a homogeneity structure will be called a \textit{graded space}. The reason for the name is the following theorem

\begin{theorem}[Grabowski-Rotkiewicz \cite{Grabowski:2012}]\label{theorem:1}
Any graded space $(F,h)$ is diffeomorphically equivalent to a \emph{dilation structure}, i.e. to a certain $(\R^\dd,h^\dd)$, where
$\dd=(d_1,\dots,d_k)$, with positive integers $d_i$,   and $\R^\dd=\R^{d_1}\times\cdots\times\R^{d_k}$
is equipped with the dilation action $h^\dd$ of multiplicative reals given by
$$h^\dd_t(y_1,\dots,y_k)=(t\cdot y_1,\dots, t^k\cdot y_k)\,,\quad y_i\in\R^{d_i}\,.$$
In other words, $F$ can be equipped with a system of (global) coordinates $(y_i^j)$, $i=1\dots, k$, $j=1,\dots, d_i$, such that $y_i^j$ is  {homogeneous of degree $i$} with respect to the homogeneity structure $h$:
$$y_i^j\circ h_t=t^i\cdot y_i^j\,.$$
Of course, in these coordinates $0^F=(0,\dots, 0)$.
\end{theorem}

It is natural to call a \textit{morphism between graded spaces} $(F_a,h^a)$, $a=1,2$, a smooth map $\Phi:F_1\to F_2$ which intertwines the homogeneity structures:
\begin{equation}\label{gr:2}\Phi\circ h_t^1=h_t^2\circ\Phi
\end{equation}

\begin{theorem}[Grabowski-Rotkiewicz \cite{Grabowski:2012}] Any morphism of graded spaces is polynomial in homogeneous fiber coordinates $y$'s. In particular the group $\Aut(\R^\dd)$ of automorphisms of $\R^\dd$ is a Lie group.
\end{theorem}
Note that automorphisms of $(\R^\dd,h^\dd)$ need not to be linear, so the category of graded spaces is different from that of vector spaces.   For instance, if $(y,z)\in\R^2$ are coordinates of degrees $1,2$, respectively, then the map
$$\R^2\ni(y,z)\mapsto (y,z+y^2)\in\R^2$$
is an automorphism of the homogeneity structure, but is nonlinear.

A straightforward generalisation is the concept of a \textit{graded bundle} $\tau:F\to M$
of rank $\dd$, with a local trivialization by $U\rightarrow\R^\dd$, and with the difference that the transition functions of local trivialisations:
$$U\cap V\times\R^\dd\ni(x,y)\longmapsto(x,A(x,y))\in U\cap V\times\R^\dd\,,$$
respect the weights of coordinates $(y^1,\dots,y^{|\dd|})$ in the fibres, i.e. $A(x,\cdot)$ are automorphisms of the graded space $(\R^\dd,h^\dd)$. In other words, a graded bundle of rank $\dd$ is a locally trivial fibration with fibers modelled on the graded space $\R^\dd$.
As these polynomials need not to be linear, graded bundles do not have, in general,
vector space structure in fibers. If all $w_i\le r$, we say that the graded bundle is \emph{of degree} $r$.

\begin{example} Consider the second-order tangent bundle $\sT^2M$, i.e. the bundle of second jets of smooth maps $(\R,0)\to M$. Writing Taylor expansions of curves in local coordinates $(x^A)$ on $M$:
\begin{equation*}x^A(t)=x^A(0)+\dot x^A(0)t+\ddot x^A(0)\frac{t^2}{2}+o(t^2)\,,\end{equation*}
we get local coordinates $(x^A,\dot x^B,\ddot x^C)$ on $\sT^2M$, which transform
\begin{eqnarray*}
x'^A&=&x'^A(x)\,,\\
\dot x'^A&=&\frac{\partial x'^A}{\partial x^B}(x)\,\dot x^B\,,\\
\ddot x'^A&=&\frac{\partial x'^A}{\partial x^B}(x)\,\ddot x^B+\frac{\partial^2 x'^A}{\partial x^B\partial x^C}(x)\,\dot x^B\dot x^C\,.
\end{eqnarray*}
This shows that associating with $(x^A,\dot x^B,\ddot x^C)$ the weights $0,1,2$, respectively, will give us a graded bundle structure of degree $2$ on $\sT^2M$. Note that, due to the quadratic terms above, this is not a vector bundle.
All this can be generalised to higher tangent bundles $\sT^kM$.
\end{example}

One can pick an atlas of the graded bundle $F$ consisting of charts for which the degrees of homogeneous local coordinates $(x^{A}, y_{w}^{a})$ are $\wu(x^{A}) =0$ and  $\wu(y_{w}^{a}) = w$, \ $1\leq w \leq k$, where $k$ is the degree of the graded bundle. The local changes of coordinates  are of the form
\begin{eqnarray}\label{eqn:translaws}
x'^{A} &=& x'^{A}(x),\\
\nonumber y'^{a}_{w} &=& y^{b}_{w} T_{b}^{\:\: a}(x) + \sum_{\stackrel{1<n  }{w_{1} + \cdots + w_{n} = w}} \frac{1}{n!}y^{b_{1}}_{w_{1}} \cdots y^{b_{n}}_{w_{n}}T_{b_{n} \cdots b_{1}}^{\:\:\: \:\:\:\:\:a}(x),
\end{eqnarray}
where $T_{b}^{\:\: a}$ are invertible and $T_{b_{n} \cdots b_{1}}^{\:\:\: \:\:\:\:\:a}$ are symmetric in indices $b_1,\dots,b_n$.

Note that the homogeneity structure in the typical fiber of a graded bundle $F$, i.e. the action $h:\R\times\R^\dd\to\R^\dd$, is preserved under the transition functions, that defines a globally defined homogeneity structure $h:\R\times F\to F$. In local homogeneous coordinates it reads
\begin{equation*}h_t(x^A,y_w^a)=(x^A,t^{w}y^a_w)\,.\end{equation*}
We call a function $f:F\to \R$  \textit{homogeneous of degree (weight) $w$} if
\begin{equation*}f\circ h_t=t^w\cdot f\,.\end{equation*}
The whole information about the degree of homogeneity is contained in the  \textit{weight vector field} (for vector bundles called the \textit{Euler vector field})
\begin{equation*}\nabla_F=\sum_swy^a_w\partial_{y^a_w}\,,\end{equation*}
so $f:F\to \R$ is homogeneous of degree $w$ if and only if $\nabla_F(f)=w\cdot f$.
Clearly, the fiber bundle morphism $\Phi$ is a smooth map which relates the weight vector fields $\nabla_{F^1}$ and $\nabla_{F^2}$.

The fundamental fact is that graded bundles and homogeneity structures are actually equivalent concepts.

\begin{theorem}[Grabowski-Rotkiewicz \cite{Grabowski:2012}]
For any homogeneity structure $h$ on a manifold $F$, there is a smooth submanifold $M$ of $F$, a non-negative integer $k\in\mathbb N$, and an $\R$-equivariant map
$\Phi_h^k:F\to \sT^kF_{|M}$
which identifies $F$ with a graded submanifold of the graded bundle $\sT^kF$. In particular, there is an atlas on $F$ consisting of local homogeneous coordinates.
\end{theorem}

The theory of vector bundles is a part of the theory of graded bundles thanks to the following.
\begin{theorem}[Grabowski-Rotkiewicz \cite{Grabowski:2012}] In the above terminology, {vector bundles are just graded bundles of degree $1$}.  The corresponding homogeneity structure is determined by the $\R$-action by homotheties. The corresponding weight vector field is the Euler vector field.
\end{theorem}

Since  {morphisms} of two homogeneity structures are defined as smooth maps $\Phi:F_1\to F_2$ intertwining the $\R$-actions: $\Phi\circ h^1_t=h^2_t\circ\Phi$, this describes also morphism of graded bundles. Consequently, a  \textit{graded subbundle} of a graded bundle $F$ is a smooth submanifold $S$ of $F$ which is invariant with respect to homotheties, $h_t(S)\subset S$ for all $t\in\R$.

\begin{definition} A  \textit{double graded bundle} is a manifold equipped with two homogeneity structures $h^1,h^2$ which are  \textit{compatible} in the sense that
$$h^1_t\circ h^2_s=h^2_s\circ h^1_t\quad \text {for all\ } s,t\in\R\,.$$
\end{definition}

The above condition can also be formulated as commutation of the corresponding weight vector fields, $[\nabla^1,\nabla^2]=0$. For vector bundles this is equivalent to the concept of a double vector bundle in the sense of Pradines.

\begin{theorem}[Grabowski-Rotkiewicz \cite{Grabowski:2006}]
The concept of a double vector bundle, understood as a particular double graded bundle in the above sense, coincides with that of {Pradines}.
\end{theorem}

All this can be extended to  {$n$-fold graded bundles} in the obvious way:

\begin{definition} An  {$n$-fold vector bundle} ({$n$-fold graded bundle}) is a manifold $F$ equipped with $n$ homogeneity structures $h^1,\dots,h^n$ of vector (graded) bundle structures which are  {compatible} in the sense that
\vskip-.3cm$$h^i_t\circ h^j_s=h^j_s\circ h^i_t\quad \text {for all\ } s,t\in\R\quad \text{and}\quad i,j=1,\dots,n\,.$$
If $h^1\circ\cdots\circ h^n(F)$ is just a single point, we speak about an $n$-fold vector (graded space).
\end{definition}

\begin{example}
If $\tau:F\to M$ is a graded bundle of degree $k$, then there are  {canonical lifts} of the graded structure to the tangent and to the cotangent bundle.   In this way $\sT F$ and $\sT^*F$ carry canonical double graded bundle structure: one is the obvious vector bundle, the other is the lifted one (of degree $k$). There are also lifts of graded structures on $F$ to $\sT^r F$.
\end{example}

Like in the $n$-tuple vector space case, the group $\G$ of automorphisms of an $n$-tuple graded space, i.e. diffeomorphisms respecting each of homogeneity structures, are polynomials of a fixed degree in $n$-homogeneous coordinates. This means that the automorphism group of an $n$-tuple graded space is a Lie group. The subgroups $G^i$  acting by identity on the homogeneous functions of $n$-degree $\ze_i=(0,\dots,0,1,0,\dots,0)\in\{ 0,1\}^n$ are normal subgroups. By a full analogy to the case of $n$-tuple vector bundles we have the following (cf. Proposition \ref{p1} and Theorem \ref{t6}).
\begin{theorem} If $(F,h^1,\dots,h^n)$ is an $n$-tuple graded space and $\G$ is the Lie group of its automorphism, then
the collection $\wG=(\G,G^1\dots,G^n)$, where $G^i$ is the subgroup of $\G$ acting identically on homogeneous functions of $n$ degree $\ze_i$, $i=1,dots,n$, is an $n$-tuple principal group.
The associated bundle construction gives
rise to an equivalence from the category of principal $\G$-bundles on $\Pe_0$ to the category
of $n$-tuple graded bundles with the model fiber $(F,h^1,\dots,h^n)$.
\end{theorem}

\small{\vskip1cm

\noindent Katarzyna GRABOWSKA\\
Faculty of Physics\\
                University of Warsaw} \\
               Pasteura 5, 02-093 Warszawa, Poland
                 \\Email: konieczn@fuw.edu.pl \\

\noindent Janusz GRABOWSKI\\ Polish Academy of Sciences\\ Institute of
Mathematics\\ \'Sniadeckich 8, 00-656 Warszawa, Poland
\\Email: jagrab@impan.pl \\

\end{document}